\documentclass[12pt]{amsart}
\usepackage{amsthm}

\usepackage{amssymb}
\usepackage{verbatim}
\usepackage[curve,matrix,arrow]{xy}
\usepackage{amsmath,amsfonts}
\usepackage{graphicx}
\usepackage{epsf}

\usepackage{amssymb, graphics}

\newcommand{\mathsym}[1]{{}}

\newcommand{\thmref}[1]{Theorem~\ref{#1}}
\newcommand{\propref}[1]{Proposition~\ref{#1}}
\newcommand{\lemref}[1]{Lemma~\ref{#1}}
\newcommand{\eqnref}[1]{Equation~(\ref{#1})}
\newcommand{\remref}[1]{Remark~\ref{#1}}

\newcommand{\figref}[1]{Figure~\ref{#1}}
\newcommand{\tabref}[1]{Table~\ref{#1}}

  {\end{list}}

\def\li{L_{i}}
\def\ri{R_{i}}

\def\NN{{\mathbb N}}

\def\XX{{\mathcal X}}
\def\SS{{\mathcal S}}

\newtheorem{theorem}{Theorem}[section]

\newtheorem{lemma}[theorem]{Lemma}
\newtheorem{proposition}[theorem]{Proposition}

\theoremstyle{example}

\newtheorem{remark}[theorem]{Remark}

\theoremstyle{definition}

\theoremstyle{notation}

\newcommand{\dd}[1]{\delta_{#1}}

\newcommand{\nn}[1]{\mu_{#1}}

\newcommand{\ga}{\Gamma}

\newcommand{\tg}{\tau(\Gamma)}

\newcommand{\ee}[1]{E(#1)}
\newcommand{\vv}[1]{V(#1)}
\newcommand{\va}{\upsilon}
\newcommand{\vb}{\text{v} \hspace{0.5 mm}}

\newcommand{\gc}{g(C)}

\newcommand{\tc}{\theta}

\newcommand{\pp}{p_{i}}

\newcommand{\qq}{q_{i}}

\newcommand{\cO}{\mathcal{O}}

\def\can{{\mathop{\rm can}}}

\def\<{\langle }
\def\>{\rangle }
\newcommand{\secref}[1]{\S\ref{#1}}

\def\ed{\epsilon_{D}}

\def\diag{\text{diag}}
\newcommand{\am}{\mathrm{A}}

\newcommand{\dm}{\mathrm{D}}

\newcommand{\jm}{\mathrm{J}}

\newcommand{\lm}{\mathrm{L}}
\newcommand{\plm}{\mathrm{L^+}}

\newcommand{\lpq}{l_{pq}}

\newcommand{\plpq}{l_{pq}^+}

\def\elg{\ell (\ga)}

\def\gc{\bar{g}}
\def\ed{\epsilon(\ga)}

\def\bq{\textbf{q}}
\def\vg{\varphi (\ga)}
\def\elg{\ell (\ga)}
\newcommand{\tcg}{\theta (\ga)}
\def\lag{\lambda (\ga)}

\newcommand\T{\rule{0pt}{2.6ex}}
\newcommand\B{\rule[-1.2ex]{0pt}{0ex}}

\begin{document}

\title[Computation of Polarized Metrized Graph Invariants]
{Computation of Polarized Metrized Graph Invariants By Using Discrete Laplacian Matrix}

\author{Zubeyir Cinkir}
\address{Zubeyir Cinkir\\
Department of Mathematics\\
Zirve University\\
27260, Gaziantep, TURKEY\\}
\email{zubeyir.cinkir@zirve.edu.tr}

%\author{Zubeyir Cinkir}
%\address{Zubeyir Cinkir\\
%Department of Mathematics\\
%University of Georgia\\
%Athens, Georgia 30602\\
%USA}
%\email{cinkir@math.uga.edu}

%\thanks{I would like to thank Dr. Robert Rumely for his continued support
%and the discussions about this paper.}

\begin{abstract}
Several invariants of polarized metrized graphs and their applications in Arithmetic Geometry are studied recently.
%\cite{Zh2}, \cite{C5}.
In this paper, we give fast algorithms to compute these invariants by expressing them in terms of the discrete Laplacian matrix and its pseudo inverse. Algorithms we give can be used for both symbolic and numerical computations. We present various examples to illustrate the implementation of these algorithms.
\end{abstract}

\maketitle

\section{Introduction}\label{sec introduction}
%\vskip .1 in

Let $X$ be a geometrically connected curve of genus $\gc \geq 2$ over a field $k$. Suppose $k$ is either a number field or the function field of a smooth projective curve $Y$ over a field. Assume that $X$ has a semistable model $\XX$ over $\SS$, where $\SS=Spec \cO_k$ if $k$ is a number field and $\SS=Y$ if $k$ is a function field. Let $N(v)$ be the local factor related to the product formula for $k$. In this context, $\omega_{\XX/\SS}^2$, the self intersection of the relative dualizing sheaf $\omega_{\XX/\SS}$, is an important quantity both in geometric and arithmetic case. Lower and upper bounds to this quantity are of interest in diophantine geometry.

In 1993, Zhang \cite{Zh1} expressed $\omega_{\XX/\SS}^2$ in terms of $\omega_{a}^2$, the self intersection of the admissible relative dualizing sheaf $\omega_{a}$ associated to a minimal regular model of $X$:
\begin{equation}\label{eqn zhang id1}
\begin{split}
\omega_{\XX/\SS}^2=\omega_a^2+\sum_{v}\epsilon(X_v) \log N(v),
\end{split}
\end{equation}
where $v$ runs over the set of non-archimedean places of $k$, and $\epsilon(X_v)$ is a certain local invariant
%of the polarized metrized graph $\Gamma_v$
of the reduction graph $R(X_v)$ associated to the completion of $X$ at a place $v$ of $k$.

In 2010, Zhang \cite[Corollary 1.3.2]{Zh2} expressed $\omega_{a}^2$ in terms of $\langle  \Delta_{\xi}, \Delta_{\xi} \rangle$,
the arithmetic self intersection (equal to the canonical height) of the Gross-Schoen cycle $\Delta_{\xi} \subset X \times X \times X$:
\begin{equation}\label{eqn zhang id2}
\begin{split}
\omega_a^2=\frac{2 \gc-2}{2 \gc+1} \langle  \Delta_{\xi}, \Delta_{\xi} \rangle +\frac{2 \gc-2}{2 \gc+1} \sum_{v}\varphi (X_{v}) \log N(v),
\end{split}
\end{equation}
where $v$ runs over the set of places of $k$, and $\varphi(X_{v})$ is a certain local invariant again linked to the reduction graph $R(X_v)$
%of a polarized metrized graph $\Gamma_{v}$ for
at each place $v$ of $k$.

On the other hand, we have the following equality for a semistable fibration $f: \XX \longrightarrow \SS$ by Nother's formula:
\begin{equation}\label{eqn Noether's formula}
\begin{split}
\omega_{\XX/\SS}^2=12 \deg f_*(\omega_{\XX/\SS})-\sum_{v} \delta(X_v),
\end{split}
\end{equation}
where $\delta(X_v)$ is the total number of the singular points in the fiber over $v$. Moreover, Zhang \cite{Zh2} showed that
\begin{equation}\label{eqn lambda and Schoen cycle}
\begin{split}
\deg f_*(\omega_{\XX/\SS}) =\frac{\gc-1}{6(2 \gc+1)} \langle  \Delta_{\xi}, \Delta_{\xi} \rangle + \sum_{v} \lambda(X_{v}) \log N(v),
\end{split}
\end{equation}
where $v$ runs over the set of places of $k$, and $\lambda(X_{v})$ is another local invariant associated to the reduction graphs $R(X_v)$.
We refer \cite{J2} for other connections between $\lambda(X_{v})$ and some invariants of complex moduli space of curves of genus higher than $1$.

Whenever $v$ is a non-archimedean place, the local invariants $\varphi(X_{v})$, $\epsilon(X_v)$, $\lambda(X_{v})$ are defined as the invariants of the corresponding polarized metrized graph. Note that the reduction graph $R(X)$ of any semistable curve $X$ of genus $\gc$ over a discrete valuation ring is a polarized metrized graph of genus $\gc$.

We also have invariants $Z(\ga)$, $\tg$ and $\tcg$ (see \secref{sec pmg and inv} below) of a polarized metrized graph $\ga$, which are closely related to the local invariants given above.

Explicit computations of these local invariants were done only for some curves of genus less than or equal to $4$ (see \cite{AM2}, \cite{AM3}, \cite{Fa}, \cite{KY1}, \cite{KY2}, \cite{KY3}, \cite{J}, \cite{C1}).
%AMoriwaki, genus 2 curve,
%KYamaki genus 2 and 3 curves,
%RD genus 2 curves for both achimedean and non-archimedean cases,
%XFaber for genus 2 and 3 cases and check for genus 4,
%Zcinkir list all relevant ones
Other than some families of polarized metrized graph of certain types, explicit computations become a huge intricate task for a person as soon as $\gc$ gets larger than 3. Thus, one needs a computer algorithm to do such computations. 
We have such an algorithm to compute $\tg$ (see \cite{C1} and \cite{C6}). Also, X. Faber gave an algorithm to compute 
the invariants $\ed$, $\vg$, $\lag$ and $Z(\ga)$ (see \cite{Fa} and \cite{Fa2}).

In this paper, we provide a fast computer algorithm that can be used for both symbolic and numeric computation of each of the polarized metrized graph invariants.
This algorithm is faster, because it takes advantage of two important facts: All of the effective resistance computations can be handled by just one matrix inversion (see \eqnref{eqn pseudo inverse} and \lemref{lem disc2}), and the computations of the invariants $\ed$, $\vg$, $\lag$ and $Z(\ga)$ can be reduced to the computations of $\tg$ and $\tcg$ (see \thmref{thm pmginv and tau}).

In \secref{sec pmg and inv}, we give a short revision of a polarized metrized graph and definitions of its invariants (see \lemref{lemtauformula}, \eqnref{eqn tcg} and \eqnref{eqn app3}). Using \thmref{thm pmginv and tau}, the problem of computing the invariants $\ed$, $\vg$, $\lag$ and $Z(\ga)$ are reduced to the computations of $\tg$ and $\tcg$. Using these connections and our previous results about $\tg$ and $\tcg$, we explain how to deal with self loops and multiple edges, if desired, as the initial step of the algorithm.

In \secref{sec discrete}, we define discrete Laplacian matrices associated to polarized metrized graphs. We express $\tcg$ in terms of the discrete Laplacian matrix $\lm$ and its pseudo inverse $\plm$ (see \thmref{thm tcg}). Then we use our previous result on $\tg$ to express
the invariants $\ed$, $\vg$, $\lag$ and $Z(\ga)$ in terms of the entries of $\lm$ and $\plm$. This gives us the main result in this paper. Namely, each of these invariants can be symbolically or numerically computed with the algorithm we provide.

Finally, we give various examples to show the implementation of our algorithm in \secref{sec examples}.

\section{Polarized Metrized graphs and their invariants}\label{sec pmg and inv}

In this section, we first give brief descriptions of a metrized graph $\ga$, a polarized metrized graph $(\ga,\bq)$, invariants $\tg$, $\tcg$, $\ed$, $\vg$, $\lag$ and $Z(\ga)$ associated to $(\ga,\bq)$.

A metrized graph $\ga$ is a finite connected graph equipped with a distinguished parametrization of each of its edges.
A metrized graph $\ga$ can have multiple edges and self-loops.
For any given $p \in \ga$,
the number $\va(p)$ of directions emanating from $p$ will be called the \textit{valence} of $p$.
%, and will be denoted by $\va(p)$.
By definition, there can be only finitely many $p \in \ga$ with $\va(p)\not=2$.

For a metrized graph $\ga$, we will denote a vertex set for $\ga$ by $\vv{\ga}$.
We require that $\vv{\ga}$ be finite and non-empty and that $p \in \vv{\ga}$ for each $p \in \ga$ if $\va(p)\not=2$. For a given metrized graph $\ga$, it is possible to enlarge the
vertex set $\vv{\ga}$ by considering additional valence $2$ points as vertices.

For a given metrized graph $\ga$ with vertex set $\vv{\ga}$, the set of edges of $\ga$ is the set of closed line segments with end points in $\vv{\ga}$. We will denote the set of edges of $\ga$ by $\ee{\ga}$. However, if
$e_i$ is an edge, by $\ga-e_i$ we mean the graph obtained by deleting the {\em interior} of $e_i$.

We define the genus of $\ga$ to be the first Betti number $g(\ga):=e-v+1$ of the graph $\ga$, where $e$ and $v$ are the number of edges and vertices of $\ga$, respectively.

We denote the length of an edge $e_i \in \ee{\ga}$ by $\li$, which represents a positive real number. The total length of $\ga$, which is denoted by $\elg$, is given by $\elg=\sum_{i=1}^e\li$.

If a metrized graph $\Gamma$ is viewed as a
resistive electric circuit with terminals at $x$ and $y$, with the
resistance in each edge given by its length, then $r(x,y)$ is
the effective resistance between $x$ and $y$ when unit current enters
at $y$ and exits at $x$.

For any $x$, $y$ in $\ga$, the resistance function $r(x,y)$ on
$\ga$ is a symmetric function in $x$ and $y$, and it satisfies
$r(x,x)=0$.
For each vertex set $\vv{\ga}$, $r(x,y)$ is
continuous on $\ga$ as a function of two variables and
%(see also \cite[Page 12]{BF}).
%As the physical interpretation suggests,
$r(x,y) \geq 0$ for all $x$, $y$ in $\ga$.
%(see also \cite[Exercise 9]{BF}).
For proofs of these facts and connections between the resistance function and the voltage function on $\ga$, see articles \cite{CR}, \cite[sec 1.5 and sec 6]{BRh}, and \cite[Appendix]{Zh1}.
The resistance function $r(x,y)$ on a metrized graph
were also studied in the articles \cite{BF} and \cite{C2}.

We will denote by $R_{i}(\ga)$, or by $R_i$ if there is no danger of confusion, the resistance between the end points of an edge $e_i$ of a graph $\ga$ when the interior of the edge $e_i$ is deleted from $\ga$.

The tau constant $\tg$ of a metrized graph $\ga$ was initially defined by Baker and Rumely \cite[Section 14]{BRh}.
The following lemma gives a description of the tau constant. In particular, it implies that the tau constant is positive.
\begin{lemma}\cite[Lemma 14.4]{BRh}\label{lemtauformula}
For any fixed $y$ in $\ga$,
%$\tg =\frac{1}{4}\int_{\ga}\big(\ddx r(x,y)\big)^2dx$.
$\tg =\frac{1}{4}\int_{\ga}\big(\frac{\partial}{\partial x} r(x,y) \big)^2dx$.
%$\tg=\frac{1}{2}\int_{\ga} r(x,y) d\mucan(y).$
\end{lemma}
One can find more detailed information on $\tg$ in articles \cite{C1}, \cite{C2}, \cite{C3} and \cite{C6}.

On a metrized graph $\ga$, we have a \textit{canonical measure} $\mu_\can$ first studied by Chinburg and Rumely \cite{CR}.
See the articles \cite{BRh} and \cite{C2} for several interpretations of $\mu_{can}$.
The following theorem gives an explicit description of $\mu_{can}$:
\begin{theorem}\cite[Theorem 2.11]{CR} \label{thmCanonicalMeasureFormula}
For a given metrized graph $\ga$, let $\li$ and $\ri$ be defined as before. Then we have
\begin{equation*}
\mu_\can(x) \ = \ \sum_{p \in \vv{\ga}} (1 - \frac{1}{2}\vb(p))
\, \delta_p(x) + \sum_{e_i \in \ee{\ga}} \frac{dx}{L_i+R_i}.
\end{equation*}
\end{theorem}

Let $\ga$ be a metrized graph and let $\bq : \ga \rightarrow  \NN$ be a function supported on the set of vertices of $\ga$.
The canonical divisor $K$ of $(\ga,\bq)$ is defined to be the following divisor on $\ga$:
\begin{equation}\label{eqn app2a}
\begin{split}
K  = \sum_{p \in \vv{\ga}} (\va(p)-2+2 \bq (p))p, \quad \text{and} \quad
 \dd{K}(x)  = \sum_{p \in \vv{\ga}} (\va(p)-2+2 \bq (p))\dd{p}(x).
\end{split}
\end{equation}
The pair $(\ga,\bq)$ is called a \textit{polarized metrized graph} (pm-graph in short) if $\bq$ is non-negative and $K$ is an effective divisor. Whenever $\bq=0$, $(\ga,\bq)$ is called a simple pm-graph. The \textit{genus} $\gc (\ga)$ of a pm-graph $(\ga,\bq)$ is defined to be
\begin{equation}\label{eqn genus}
\begin{split}
\gc (\ga) = 1+\frac{1}{2}\deg{K}=g(\ga)+\sum_{p \in \vv{\ga}}\bq (p).
\end{split}
\end{equation}
We will simply use $\gc$ instead of $\gc(\ga)$ when there is no danger of confusion.

Let $\nn{ad}(x)$ be the admissible measure associated to $K$ (defined by Zhang \cite[Lemma 3.7]{Zh1}). We have
$$\nn{ad}(x) = \frac{1}{\gc} \Big(\sum_{p \in \vv{\ga}}\bq (p) \dd{p}(x) +
\sum_{e_i \in \ee{\ga}} \frac{dx}{L_{i}+R_{i}} \Big).$$
Then, if we use \thmref{thmCanonicalMeasureFormula}, we can relate $\nn{ad}(x)$ and $\nn{can}(x)$ as follows:
\begin{equation}\label{eqn app2b}
\begin{split}
\nn{ad}(x) = \frac{1}{2 \gc}(2\nn{can}(x)+\dd{K}(x)).
\end{split}
\end{equation}
Moreover, $\dd{K}(\ga) =  \deg(K)  =  2 \gc-2$, and we have $\nn{can}(\ga) =  1 = \nn{ad}(\ga)$.

On a pm-graph $(\ga,\bq)$, we defined and studied (\cite{C1} and \cite{C5}) the invariant $\tcg$ as follows:
%$$\tcg:=\sum_{p, \, q \in \, \vv{\ga}}(\va(p)-2+2 \bq (p))(\va(q)-2+2 \bq (q))r(p,q).$$
\begin{equation}\label{eqn tcg}
\begin{split}
\tcg:=\sum_{p, \, q \in \, \vv{\ga}}(\va(p)-2+2 \bq (p))(\va(q)-2+2 \bq (q))r(p,q).
\end{split}
\end{equation}
We have $\tcg \geq 0$ for any pm-graph $\ga$, since the canonical divisor $K$ is effective.

Now, we can give definitions of the invariants $\ed$, $\vg$, $\lag$ and $Z(\ga)$ (c.f. \cite[Section 4.1]{Zh2}) of $\ga$:
\begin{equation}\label{eqn app3}
\begin{aligned}
\ed &=\iint_{\Gamma \times \Gamma} r(x,y) \dd{K}(x)  \nn{ad}(x), & \quad Z(\ga) &=\frac{1}{2} \iint_{\Gamma \times \Gamma} r(x,y) \nn{ad}(x) \nn{ad}(y),
\\
\vg &=3 \gc \cdot Z(\ga) -\frac{1}{4} (\ed +\ell (\ga)), & \quad \lag &=\frac{\gc-1}{6 (2 \gc+1)} \vg +\frac{1}{12}(\ed +\elg).
\end{aligned}
\end{equation}
%
%\begin{proposition}\label{prop ed and tau}
%Let $\ga$ be a pm-graph. Then we have
%$
%\ed = \frac{(4 \gc-4) \tg}{\gc} + \frac{\tcg}{2 \gc}.
%$
%\end{proposition}
%
%\begin{proposition}\label{prop adm int of res}
%Let $\ga$ be a pm-graph. Then we have
%$
%Z(\ga) = \frac{(2 \gc -1)\tg}{\gc^2}+\frac{\tc(\ga)}{8 \gc^2}.
%$
%\end{proposition}
%
%\begin{theorem}\label{thm main1}
%Let $\ga$ be a pm-graph. Then we have
%$
%\vg = \frac{(5 \gc -2) \tg}{\gc}+\frac{\tcg}{4 \gc}-\frac{\elg}{4}.
%$
%\end{theorem}
%
%\begin{proposition}\label{propcor lambda interms of tau}
%Let $\ga$ be a pm-graph. Then we have
%\begin{equation*}
%\begin{split}
%\lag = \frac{(3\gc -3) \tg}{4 \gc+2}+\frac{\tcg}{16 \gc+8}+\frac{(\gc +1)\elg}{16 \gc+8}.
%\end{split}
%\end{equation*}
%\end{proposition}

We can express invariants given in \eqnref{eqn app3} in terms of $\tg$ and $\tcg$ (\cite[Propositions 4.6, 4.7, 4.9 and Theorem 4.8]{C5}):
\begin{theorem}\label{thm pmginv and tau}
Let $(\ga,\bq)$ be a pm-graph. Then we have
\begin{align*}\label{eqn pmginv and tau}
\vg &= \frac{(5 \gc -2) \tg}{\gc}+\frac{\tcg}{4 \gc}-\frac{\elg}{4},&  Z(\ga) &= \frac{(2 \gc -1)\tg}{\gc^2}+\frac{\tc(\ga)}{8 \gc^2},
\\ \lag &= \frac{(3\gc -3) \tg}{4 \gc+2}+\frac{\tcg}{16 \gc+8}+\frac{(\gc +1)\elg}{16 \gc+8},&  \ed & = \frac{(4 \gc-4) \tg}{\gc} + \frac{\tcg}{2 \gc}.
\end{align*}
\end{theorem}

\begin{remark}\label{rem valence}
Given a pm-graph $(\ga,\bq)$ and $s \in \ga-\vv{\ga}$, we have $\bq (s)=0$ and $\va(s)=2$. If we enlarge the vertex set by
considering $s$ as a vertex with the same $\bq$ value on $s$, $\tcg$ does not change. We recall that this process also does not change the value of $\tg$ (the valence property of $\tg$, see \cite[Remark 2.10]{C2}). Conversely, if $\bq (s)=0$ and $\va(s)=2$ for some $s \in \vv{\ga}$ with $\vv{\ga}$ has at least two elements, then removing $s$ from the vertex set of $\ga$ does not change $\tg$ and $\tcg$ (such vertices are called eliminable vertices in \cite[pg. 152]{KY1}).
Using these observations along with \thmref{thm pmginv and tau}, we note that
$\ed$, $Z(\ga)$, $\vg$ and $\lag$ do not change under this process. That is, each of these invariants has the valence property.
\end{remark}

Let $(\ga,\bq)$ be a polarized metrized graph containing a self loop of length $L$ at a vertex $p$ with $\va(p)\geq 3$. Let $\beta$ be a metrized graph obtained from $\ga$ by deleting this self loop. We still have $p \in \vv{\beta}$ but its valence is reduced by $2$. Now, we consider a function $\bar{\bq} : \beta \rightarrow  \NN$ supported on the set of vertices of $\beta$ that satisfies $\bar{\bq}(p)=\bq(p)+1$ and  $\bar{\bq}(x)=\bq(x)$ for every $x \in \beta-\{p\}$. Note that $(\beta,\bar{\bq})$ is a polarized metrized graph, and that $\gc(\ga)=\gc(\beta)$, $\ell(\ga)=\ell(\beta)+L$, $\tcg=\theta(\beta)$ and $\tg=\tau(\beta)+\frac{L}{12}$ by the additive property of $\tg$ \cite[pg. 15]{C2} and by the fact that a self loop of length $L$ has the tau constant value $\frac{L}{12}$. Using these facts along with \thmref{thm pmginv and tau}, we obtain the following equalities:

\begin{equation}\label{eqn pmginv and self loop}
\begin{aligned}
\vg &= \varphi(\beta)+\frac{\gc -1}{6 \gc}L,  & Z(\ga) &= Z(\beta)+\frac{2 \gc-1}{12 \gc ^2}L,
\\ \lag &= \lambda(\beta)+\frac{\gc}{8 \gc+4}L, & \ed &= \epsilon(\beta)+\frac{ \gc-1}{3 \gc}L.
\end{aligned}
\end{equation}
Successive application of the process above gives a pm-graph $\beta$ that is either a graph with one vertex and a self loop at this vertex or a graph without self loops. For the first case, we can use the following proposition (part of which is nothing but \cite[Proposition 4.4.3]{Zh2}. Here we give a new proof.):
%(which is nothing but \cite[Prop 5.1]{C5} if $\bq =0$) below.
\begin{proposition}\label{prop bouquet}
Let $(\ga,\bq)$ be a pm-graph with one vertex and $e \geq 1$ self loops at its vertex. Then
\begin{equation*}\label{eqn bouquet1}
\begin{aligned}
\vg &= \frac{\gc -1}{6 \gc}\elg,  & Z(\ga) &= \frac{2 \gc-1}{12 \gc ^2}\elg,
\\ \lag &= \frac{\gc}{8 \gc+4}\elg, & \ed &= \frac{ \gc-1}{3 \gc}\elg.
\end{aligned}
\end{equation*}
\end{proposition}
\begin{proof}
If $\ga$ has just one self loop, then $\tg =\frac{\elg}{12}$ and $\tcg=0$. Thus, the result follows from \thmref{thm pmginv and tau}
in this case. Applying the procedure given in \eqnref{eqn pmginv and self loop}, one can show that the result still holds if $\ga$ has more than one self loop.
\end{proof}

Hence, we can focus on pm-graphs with no self loop.

\section{Discrete Laplacian}\label{sec discrete}

In this section, we first describe the discrete Laplacian matrix of a finite weighted graph with no self loops and multiple edges. Then following \cite{C6}, we define discrete Laplacian matrices associated to a metrized graph $\ga$ and a pm-graph $(\ga,\bq)$. Then we express invariants of a pm-graph in terms of the associated discrete Laplacian matrices, their pseudo inverse and values of $\bq$. This enables us to achieve our main goal in this paper. Namely, we derive a feasible algorithm for both numeric and symbolic computations of pm-graph invariants.

To have a well-defined discrete Laplacian matrix $\lm$ for a metrized
graph $\ga$, we first choose a vertex set $\vv{\ga}$ for $\ga$ in
such a way that there are no self-loops, and no multiple edges
connecting any two vertices. This can be done by
enlarging the vertex set by considering additional valence two points as vertices
whenever needed. We call such a vertex set $\vv{\ga}$ \textit{adequate}.
If distinct vertices $p$ and $q$ are the end points of an edge, we
call them \textit{adjacent} vertices.

%\begin{definition}\label{def laplacian}
Let $\ga$ be a metrized graph with $e$ edges and an adequate vertex set
$\vv{\ga}$ containing $v$ vertices. Fix an ordering of the vertices
in $\vv{\ga}$. Let $\{L_1, L_2, \cdots, L_e\}$ be a labeling of the
edge lengths. The matrix $\am=(a_{pq})_{v \times v}$ given by
\[
a_{pq}=\begin{cases} 0, & \quad \text{if $p = q$, or $p$ and $q$ are
not adjacent}.\\
\frac{1}{L_k}, & \quad \text{if $p \not= q$, and an edge of length $L_k$ connects $p$ and $q$.}\\
%$p$ and $q$ are connected by} \text{ an edge of length $L_k$}\\
\end{cases}
\]
is called the \textit{adjacency matrix} of $\ga$. Let $\dm=\diag(d_{pp})$ be
the $v \times v$ diagonal matrix given by $d_{pp}=\sum_{s \in
\vv{\ga}}a_{ps}$. Then $\lm:=\dm-\am$ is called the \textit{discrete
Laplacian matrix} of $\ga$. That is, $ \lm =(l_{pq})_{v \times v}$ where
\[
l_{pq}=\begin{cases} 0, & \; \, \text{if $p \not= q$, and $p$ and $q$
are not adjacent}.\\
-\frac{1}{L_k}, & \; \, \text{if $p \not= q$, and $p$ and $q$ are
connected by} \text{ an edge of length $L_k$}\\
-\sum_{s \in \vv{\ga}-\{p\}}l_{ps}, & \; \, \text{if $p=q$}
\end{cases}.
\]
%\end{definition}

We define the discrete Laplacian matrix corresponding to pm-graph $(\ga,\bq)$ as the
discrete Laplacian $\lm$ corresponding to the metrized graph $\ga$. Important thing is
that if the vertex set of $(\ga,\bq)$ is needed to be enlarged to make it an adequate
vertex set we set $\bq$ value $0$ for those new vertices added because of self loops
or multiple edges.

Although $\lm$ is not invertible, it has generalized inverses. In
particular, it has the pseudo inverse $\plm$, also known as the
Moore-Penrose generalized inverse. The pseudo inverse $\plm$ is uniquely determined by
the following properties:
\begin{align*}
i)  \quad &\lm \plm \lm  = \lm, & \qquad \qquad
iii) \quad &(\lm \plm)^{T}   = \lm \plm,
\\ ii) \quad &\plm \lm \plm  = \plm, & \qquad \qquad
iv) \quad &(\plm \lm)^{T}   = \plm \lm .
\end{align*}

For a discrete Laplacian matrix $\lm$ of size $v \times v$, the following
formula for $\plm$ (see \cite[ch 10]{C-S}) is crucial for our computations:
\begin{equation}\label{eqn pseudo inverse}
 \plm = \big( \lm - \frac{1}{v}\jm \big)^{-1} + \frac{1}{v} \jm.
\end{equation}
where $\jm$ is of size $v \times v$ and has all entries $1$.

$\lm$ and $\plm$ are symmetric matrices. 
%and $\lm \text{ and } \plm \text{ are doubly centered}$.
\begin{remark}\label{rem doubcent1}
We have  $\sum_{p \in \vv{\ga}} \lpq = 0 =\sum_{p \in \vv{\ga}} \plpq $,
for each $q \in \vv{\ga}$. 
%Also, $\plpq = l_{qp}^+$, for each $p$, $q$ $\in \vv{\ga}$.
\end{remark}

\begin{lemma} \cite{RB2}, \cite{RB3}, \cite[Theorem A]{D-M} \label{lem disc2}
Suppose $\ga$ is a graph with the discrete Laplacian $\lm$ and the
resistance function $r(x,y)$.
%Let $\hm$ be a generalized inverse
%of $\lm$ (i.e., $\lm \hm \lm = \lm$). Then we have
%$$r(p,q)=\hm_{pp}-\hm_{pq}-\hm_{qp}+\hm_{qq}, \quad \text{for any $p$, $q$ $\in \vv{\ga}$}.$$
%In particular, 
For the pseudo inverse $\plm$ we have
$$r(p,q)=l_{pp}^+-2l_{pq}^+ + l_{qq}^+, \quad \text{for any $p$, $q$ $\in \vv{\ga}$}.$$
\end{lemma}
It is important that $\tg$ can be expressed in terms of the discrete Laplacian matrix and its pseudo inverse:
\begin{theorem}\cite[Theorem 4.10]{C6}\label{thm disc2}
Let $\lm$ be the discrete Laplacian matrix of size $v \times v$ for a metrized graph
$\ga$, and let $\plm$ be its Moore-Penrose pseudo inverse. Suppose $\pp$ and $\qq$ are end points of $e_i \in \ee{\ga}$. Then we have
\begin{equation*}
\begin{split}
\tg & =-\frac{1}{12}\sum_{e_i \in \ee{\ga}} l_{\pp \qq}\big(\frac{1}{ l_{\pp \qq}}+l_{\pp \pp}^+-2l_{\pp
\qq}^+ +l_{\qq \qq}^+ \big)^2 +\frac{1}{4}\sum_{q, \, s  \in
\vv{\ga}} l_{qs} l_{qq}^+  l_{ss}^+ +\frac{1}{v}tr(\plm).
%\\ \tg & = -\frac{1}{12}\sum_{e_i \in \ee{\ga}} l_{\pp \qq}\big(\frac{1}{ l_{\pp \qq}}+l_{\pp \pp}^+-2l_{\pp
%\qq}^+ +l_{\qq \qq}^+ \big)^2 -\frac{1}{4}\sum_{e_i \in \ee{\ga}} l_{p_i q_i} \big(l_{\pp \pp}^+ -  l_{\qq \qq}^+\big)^2 %+\frac{1}{v}tr(\plm).
\end{split}
\end{equation*}
Moreover, $\sum_{q, \, s  \in
\vv{\ga}} l_{qs} l_{qq}^+  l_{ss}^+ = - \sum_{e_i \in \ee{\ga}} l_{p_i q_i} \big(l_{\pp \pp}^+ -  l_{\qq \qq}^+\big)^2 $.
\end{theorem}
Similarly, one can express $\mu_{can}$ in terms of the discrete Laplacian matrix and its pseudo inverse:
\begin{proposition}\cite[Prop. 4.12]{C6}\label{prop mucan}
For a given metrized graph $\ga$, let $\lm$ be its discrete Laplacian, and let $\plm$ be the corresponding pseudo inverse.
Suppose $\pp$ and $\qq$ are end points of $e_i \in \ee{\ga}$. Then we have
$$\mu_\can(x) = \sum_{p \in \vv{\ga}} (1 - \frac{1}{2}\vb(p))
\, \delta_p(x) - \sum_{e_i \in \ee{\ga}} \big(l_{\pp \qq}+l_{\pp \qq}^2 (l_{\pp \pp}^+ -2 l_{\pp \qq}^+ +l_{\qq \qq}^+)\big)dx.$$
\end{proposition}
Using \eqnref{eqn app2b} and \propref{prop mucan}, we express $\mu_{ad}(x)$ in terms of the entries of $\lm$ and $\plm$:
%the discrete Laplacian matrix and its pseudo inverse:
\begin{proposition}\label{prop muad}
For a given pm-graph $(\ga,\bq)$, let $\lm$ be its discrete Laplacian, and let $\plm$ be the corresponding pseudo inverse.
Suppose $\pp$ and $\qq$ are end points of $e_i \in \ee{\ga}$. Then we have
$$\mu_{ad}(x) = \frac{1}{\gc} \sum_{p \in \vv{\ga}} \bq(p)\, \delta_p(x)
- \frac{1}{\gc} \sum_{e_i \in \ee{\ga}} \big(l_{\pp \qq}+l_{\pp \qq}^2 (l_{\pp \pp}^+ -2 l_{\pp \qq}^+ +l_{\qq \qq}^+)\big)dx.$$
\end{proposition}

The next important observation is that $\tcg$ can be expressed in terms of $\lm$ and $\plm$:
%the discrete Laplacian matrix and its pseudo inverse:
\begin{theorem}\label{thm tcg}
Let $(\ga, \bq)$ be a pm-graph. Then
$$\tcg=2 (2 \gc-2)\sum_{p \in \, \vv{\ga}} (\va(p)-2+2 \bq (p)) l_{p p}^+
-2\sum_{p, \, q \in \, \vv{\ga}}(\va(p)+2 \bq (p))(\va(q)+2 \bq (q)) l_{p q}^+.$$
\end{theorem}
\begin{proof}
Since $\deg(K)  =  2 \gc-2$, $\sum_{s \in \, \vv{\ga}} (\va(s)-2+2 \bq (s))=2 \gc-2$. Then the result follows from
\eqnref{eqn tcg}, \lemref{lem disc2} and \remref{rem doubcent1}.
\end{proof}

Whenever $\bq =0$ on $\ga$, i.e. $\ga$ is a simple graph, we have $\gc=g=e-v+1$ and
\begin{equation}\label{eqn teta simple}
\tcg=2 (2 g-2)\sum_{p \in \, \vv{\ga}} (\va(p)-2) l_{p p}^+
-2\sum_{p, \, q \in \, \vv{\ga}}\va(p) \va(q) l_{p q}^+.
\end{equation}
Moreover, if $\ga$ is both simple and $r$-regular, we have $e=\frac{r}{2}v$, $\gc=\frac{r-2}{2}v+1$ and we have
\begin{equation}\label{eqn teta simple regular}
\tcg=2 v (r-2)^2 tr(\plm),
\end{equation}
since $\sum_{p \in \, \vv{\ga}} l_{p q}^+=0$.

%Next, we derive our main result in this paper.
%Next, we give the pseudo code for our algorithm. We use \thmref{thm pmginv and tau}, \thmref{thm disc2} and \thmref{thm tcg} to express the %invariants $\ed$, $\vg$, $\lag$ and $Z(\ga)$ in terms of the discrete Laplacian matrix $\lm$ and its pseudo inverse $\plm$.

Given a pm-graph $(\ga,\bq)$, we can give the steps of the algorithm to compute the invariants $\ed$, $\vg$, $\lag$ and $Z(\ga)$ as follows:

\vskip 0.1 in
\Small
%\begin{verbatim}
%
%\textbf{procedure} computation of the invariants $\tg$, $\tcg$, $\ed$, $\vg$, $\lag$ and $Z(\ga)$  (\textbf{$(\ga,\bq)$}: is the given %pm-graph)
Set $\lm$ be the discrete Laplacian matrix corresponding to $\ga$. (Choose an adequate vertex set so that there will be no self loops or no multiple edges. Alternatively, use \eqnref{eqn pmginv and self loop} for self loops.)

Set $v$ be the number of rows in $\lm$ (i.e., the number of vertices in $\ga$).

Set $e$ be the number of nonzero entries above the diagonal of $\lm$ (i.e., the number of edges in $\ga$).

Set $\elg$ be the negative of the sum of the reciprocals of nonzero entries above the diagonal of $\lm$ (i.e., the total length of $\ga$).

Compute $\gc$ by using \eqnref{eqn genus}.

Compute $\plm$, pseudo inverse of $\lm$, by using the formula given in \eqnref{eqn pseudo inverse}.

Compute $\tg$ by using \thmref{thm disc2}.

Define a function $\va$ whose value at $i$ is, $\va(i)$, the number of nonzero off diagonal entries in the $i$-th row of $\lm$ (i.e., $\va(i)$ is the valence of the vertex $i$).

Compute  $\tcg$ by using \thmref{thm tcg}.

Compute $\ed$, $\vg$, $\lag$ and $Z(\ga)$ by using \thmref{thm pmginv and tau}.

%
%function R=Tau(A)
%
%kk=size(A);
%k=kk(1);
%B=inv(A-(1/k)*ones(k));
%R=0;
%for v=1:k
%    for w=v+1:k
%        if A(v,w)<0
%            Li=-1/A(v,w);
%            S=B(v,v)+B(w,w)-2*B(v,w);
%            R= R+1/12*((Li-S)^2)/Li-(S/(4*Li))*(S+2*B(v,w)+2/k);
%        end
%    end
%end
%R=R+(1+trace(B))*(k+1)/(2*k);
%
%
%\textbf{for} i:=1 \textbf{to} k
%
%\qquad $C=F(C)$;
%
%\textbf{end}
%
%\end{verbatim}
\normalsize
\vskip 0.1 in

Note that the most costly part of this algorithm is to obtain $\plm$ from $\lm$.

\section{Examples}\label{sec examples}

In this section, we give various examples illustrating the implementation of the algorithm described in the previous section.

A symbolic computation is exemplified as follows:

\textbf{Example 1:}

Let $\ga$ be a simple pm-graph as illustrated in \figref{fig completeg4}. In this case, we have
$\elg=\frac{1}{a}+\frac{1}{b}+\frac{1}{c}+\frac{1}{d}+\frac{1}{e}+\frac{1}{f}$, and the discrete Laplacian
matrix corresponding to $\ga$ is given as follows:
\[
\lm=\left[
\begin{array}{cccc}
 a+b+c & -a & -b & -c  \\
 -a & a+d+e & -d & -e  \\
 -b & -d & b+d+f & -f  \\
 -c & -e & -f & c+e+f
\end{array}
\right].
\]
\begin{figure}
\centering
\includegraphics[scale=0.5]{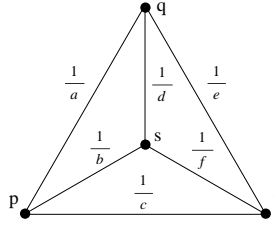} \caption{Complete graph with vertices $\{ p, q, s, t \}$ and edges of lengths $\{\frac{1}{a}, \frac{1}{b}, \frac{1}{c}, \frac{1}{d}, \frac{1}{e}, \frac{1}{f} \}$.} \label{fig completeg4}
\end{figure}
We first compute $\plm$ by using \eqnref{eqn pseudo inverse}, then use the algorithms given above to derive the following results:
\begin{align*}\label{eqn pminv of completeg4}
\tg &=\frac{1}{12}\elg-\frac{A+2 B}{6 C}, &  \tcg&=\frac{6A+8B}{C},
\\ \vg&=\frac{1}{9}\elg-\frac{2A+7B}{9 C},&  \lag&=\frac{3}{28}\elg+\frac{A}{28 C},
\\ Z(\ga)&=\frac{5}{108}\elg-\frac{A+8B}{108 C},&  \ed&=\frac{2}{9}\elg+\frac{5A+4B}{9 C},
\end{align*}
where $A=a b + a c + b c + a d + b d + a e + c e + d e +
 b f + c f + d f + e f$, $B=c d + b e + a f$, $C=a b c + a c d + b c d + a b e + b c e + a d e + b d e + c d e +
 a b f + a c f + a d f + b d f + c d f + a e f + b e f + c e f$.

The following example has mostly numeric computations and some symbolic computations because of non-zero $\bq$.

\textbf{Example 2:}

Let $\ga$ be a pm-graph $(\ga,\bq)$ such that $\ga$ is the complete graph on $4$ vertices $\{ p, q, s, t \}$ as in \figref{fig completeg4}. Suppose that each edge of $\ga$ has length $\frac{1}{6}$ and that $\bq (p)=\bq (q)=\bq (s)=\bq (t)=k$ for some nonnegative integer $k$. In this case, we have $\elg=1$, the following discrete Laplacian matrix and its pseudo inverse:
\[
\lm=\left[
\begin{array}{cccc}
 18 & -6 & -6 & -6  \\
 -6 & 18 & -6 & -6  \\
 -6 & -6 & 18 & -6  \\
 -6 & -6 & -6 & 18
\end{array}
\right], \quad
\plm=\left[
\begin{array}{cccc}
 \frac{1}{32} & \frac{-1}{96} & \frac{-1}{96} & \frac{-1}{96}   \medskip\\
 \frac{-1}{96} & \frac{1}{32} & \frac{-1}{96} & \frac{-1}{96}   \medskip\\
 \frac{-1}{96} & \frac{-1}{96} & \frac{1}{32} & \frac{-1}{96}   \medskip\\
 \frac{-1}{96} & \frac{-1}{96} & \frac{-1}{96} & \frac{1}{32}
\end{array}
\right].
\]
Then we obtain these results:
\begin{align*}\label{eqn pminv of completeg4}
\tg &=\frac{5}{96}, &  \tcg&=(1 + 2 k)^2, & \vg &=\frac{96 k^2+100 k+17}{96 (4 k+3)},
\\ \lag&=\frac{16 k^2+ 42 k+25}{32 ( 8 k+7)},&  Z(\ga)&=\frac{48 k^2+88k+37}{96(4k+3)^2}, & \ed &=\frac{(12k+11)(2k+1)}{12(4k+3)}.
\end{align*}

In particular, if $k=0$, $\tcg=1$, $\vg=\frac{17}{288}$, $\lag=\frac{25}{224}$, $Z(\ga)=\frac{37}{864}$ and $\ed=\frac{11}{36}$.

Sometimes we are given pm-graphs containing self loops or multiple edges. We exemplify our strategy for such cases in detail as follows:

\textbf{Example 3:}

In this example, we consider a pm-graph $(\ga,\bq)$ as the graph $I$ in \figref{fig ex3}, where the edge lengths and the values of $\bq$ are illustrated. In this case, $\ga$ has a self loop and two multiple edges. We first ignore the vertex with $\bq$ value $0$ and valence $2$. To avoid having multiple edges, we add a vertex to one of the multiple edges such that the added vertex has $\bq$ value $0$. We can deal with the self loop by either applying \eqnref{eqn pmginv and self loop} or considering two additional vertex with $\bq$ value $0$ on the self loop. These cases are illustrated by graphs $II$ and $III$ in \figref{fig ex3}. Thus, we can compute polarized metrized graph invariants of $\ga$ in two different ways, and the results are as follows:
\begin{align*}
\tg &=\frac{b + 2 c + 3 (a+A)}{12}, & \tcg &=\frac{480 b c}{b + 2 c}+170A,\\
\vg &=\frac{11a}{24}+\frac{11 b^2 + 764 b c +44c^2 }{72 (b + 2 c)}+\frac{9A}{2}, & \lag &=\frac{9a}{25}+\frac{3 (b^2 + 24 b c + 4 c^2)}{25 (b + 2 c)}+\frac{27A}{25},\\
\ed &=\frac{11a}{12}+\frac{11 b^2 + 764 b c + 44 c^2}{36 (b + 2 c)}+8A,& Z(\ga) &=\frac{23a}{576}+\frac{23 b^2 + 812 b c + 92 c^2}{1728 (b + 2 c)}+\frac{3A}{16},
\end{align*}
where $A=d+e$, $\elg=3a+b + 2 c + d + e$, $ g(\ga)=2$ and $\gc(\ga)=12$.

%\begin{align*}
%\tg &=\frac{b + 2 c + 3 (a+d + e)}{12}, & \tcg &=\frac{480 b c}{b + 2 c}+170(d+e)\\
%\vg &=\frac{11a}{24}+\frac{11 b^2 + 764 b c +44c^2 }{72 (b + 2 c)}+\frac{9(d+e)}{2}, & \lag &=\frac{9a}{25}+\frac{3 (b^2 + 24 b c + 4 %c^2)}{25 (b + 2 c)}+\frac{27(d+e)}{25}\\
%\lag &=\frac{9a}{25}+\frac{3 (b^2 + 24 b c + 4 c^2)}{25 (b + 2 c)}+\frac{27(d+e)}{25},& &\\
%\ed &=\frac{11a}{12}+\frac{11 b^2 + 764 b c + 44 c^2}{36 (b + 2 c)}+8(d+e),& &\\
%Z(\ga) &=\frac{23a}{576}+\frac{23 b^2 + 812 b c + 92 c^2}{1728 (b + 2 c)}+\frac{3(d+e)}{16}, & &
%\end{align*}
%where $\elg=3a+b + 2 c + d + e$, $ g(\ga)=2$ and $\gc(\ga)=12$.

\begin{figure}
\centering
\includegraphics[scale=0.7]{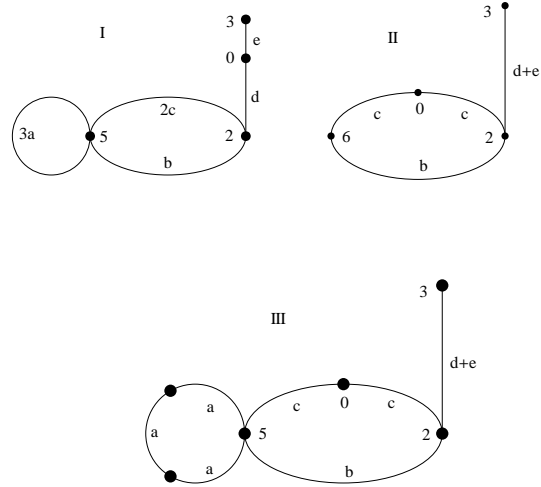} \caption{A pm-graph with a self loop, nonzero $\bq$ and two multiple edges.} \label{fig ex3}
\end{figure}

The following example is about the computation of invariants of a class of simple pm-graphs. It is given to show both symbolic computations and different aspects of numerical computations which would be critical during the implementation of our algorithm.

\textbf{Example 4:}

A ladder graph $L_n(a,b)$ is planar graph that looks like a ladder with $n$ rungs.
It has $2n$ vertices and $3n-2$ edges. Edges looking like rungs are of lengths $b$, and each of the remaining $2(n-1)$ edges has length $a$. Thus, $\ell(L_n(a,b))=2(n-1)a+n b$ and $g(\ga)=n-1$.
\figref{fig ladder graph} shows an example. Using the same notation for the corresponding simple pm-graph, we obtain the results given in the tables below:
\begin{figure}
\centering
\includegraphics[scale=0.35]{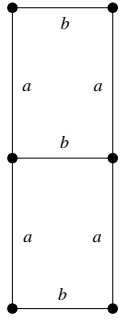} \caption{Ladder graph $L_3(a,b)$.} \label{fig ladder graph}
\end{figure}

\begin{table}
\begin{center}
\begin{tabular}{|c|c|c|c|c|}
  \hline
 $n$  & e & v & $g(L_n(a,b))$ &  $\ell(L_n(a,b))$\\
 \hline
 $2$ & 4 & 4 & 1 & $2 (a + b)$\\
 \hline
 $3$ & 7 & 6 & 2 & $4 a + 3 b$\\
 \hline
 $4$ & 10 & 8 & 3 & $2 (3 a + 2 b)$ \\
 \hline
 $5$ & 13 & 10 & 4 & $8 a + 5 b$\\
 \hline
\end{tabular}
\end{center} \caption{Number of edges, vertices, genus and length of $L_n(a,b)$ for $n \in \{ 2, \, 3, \, 4, \, 5 \}$.} \label{table length ladder}
\end{table}
We used Mathematica \cite{MMA} to do the symbolic computations given in Tables \ref{table tau ladder}, \ref{table phi ladder} and \ref{table epsilon ladder}. As the number of vertices gets larger the expressions in the results become more complicated, so we listed the results only for small values of $n$.
\begin{table}
%\footnotesize
\begin{center}
\begin{tabular}{|c|c|c|}
  \hline
 $n$ \T \B &  $\frac{\tg}{\elg}$ &  $\frac{\tcg}{\elg)}$ \\
 \hline
 $2$ \T \B &  $\frac{1}{12}$ & $0$ \\
 \hline
 $3$ \T \B & $\frac{8 a^2 + 14 a b + 7 b^2}{12 (2 a + 3 b) (4 a + 3 b)}$ & $\frac{2 b (2 a + b)}{(2 a + 3 b) (4 a + 3 b)}$ \\
 \hline
 $4$ \T \B & $\frac{12 a^4 + 36 a^3 b + 38 a^2 b^2 + 18 a b^3 +
 3 b^4}{24 (a + b) (3 a + 2 b) (2 a^2 + 4 a b + b^2)}$ & $\frac{4 a^4 + 20 a^3 b + 26 a^2 b^2 + 10 a b^3 + b^4}{(a + b) (3 a + 2 b) (2 a^2 + 4 a b + b^2)}$ \\
 \hline
 $5$ \T \B & $\frac{128 a^5 + 496 a^4 b + 704 a^3 b^2 + 476 a^2 b^3 + 160 a b^4 +
 19 b^5}{12 (8 a + 5 b) (4 a^2 + 6 a b + b^2) (4 a^2 + 10 a b +
   5 b^2)}$ & $\frac{2 (128 a^5 + 656 a^4 b + 1104 a^3 b^2 + 716 a^2 b^3 + 160 a b^4 +
   9 b^5)}{(8 a + 5 b) (4 a^2 + 6 a b + b^2) (4 a^2 + 10 a b + 5 b^2)}$ \\
 \hline
\end{tabular}
\end{center} \caption{For pm-graph $\ga=L_n(a,b)$, $\tg$ and $\tcg$ for $n \in \{ 2, \, 3, \, 4, \, 5 \}$.} \label{table tau ladder}
%\normalsize
\end{table}
\begin{table}
\begin{center}
\begin{tabular}{|c|c|c|}
  \hline
 $n$ \T \B &  $\frac{\vg}{\elg}$ &  $\frac{\lag}{\elg}$ \\
 \hline
 $2$ \T \B &  $0$ & $\frac{1}{12}$\\
 \hline
 $3$ \T \B & $\frac{2 a^2 + 2 a b + b^2}{3 (2 a + 3 b) (4 a + 3 b)}$ & $\frac{1}{10}$\\
 \hline
 $4$ \T \B & $\frac{72 a^4 + 192 a^3 b + 164 a^2 b^2 + 60 a b^3 +
 9 b^4}{72 (a + b) (3 a + 2 b) (2 a^2 + 4 a b + b^2)}$ & $\frac{5 a + 3 b}{14 (3 a + 2 b)}$ \\
 \hline
 $5$ \T \B & $\frac{(2 a + b) (16 a^4 + 52 a^3 b + 51 a^2 b^2 + 16 a b^3 +
   2 b^4)}{(8 a + 5 b) (4 a^2 + 6 a b + b^2) (4 a^2 + 10 a b + 5 b^2)}$ & $\frac{5 (2 a + b)}{9 (8 a + 5 b)}$\\
 \hline
\end{tabular}
\end{center} \caption{For pm-graph $\ga=L_n(a,b)$, $\vg$ and $\lag$ for $n \in \{ 2, \, 3, \, 4, \, 5 \}$.} \label{table phi ladder}
\end{table}
\begin{table}
\begin{center}
\begin{tabular}{|c|c|c|}
  \hline
 $n$ \T \B &  $\frac{\ed}{\elg}$ &  $\frac{Z(\ga)}{\elg}$ \\
 \hline
 $2$ \T \B &  $0$ & $\frac{1}{12}$\\
 \hline
 $3$ \T \B & $\frac{4 a^2 + 10 a b + 5 b^2}{3 (2 a + 3 b) (4 a + 3 b)}$ & $\frac{(a + b)^2}{2 (2 a + 3 b) (4 a + 3 b)}$\\
 \hline
 $4$ \T \B & $\frac{36 a^4 + 132 a^3 b + 154 a^2 b^2 + 66 a b^3 +
 9 b^4}{18 (a + b) (3 a + 2 b) (2 a^2 + 4 a b + b^2)}$ & $\frac{36 a^4 + 120 a^3 b + 134 a^2 b^2 + 60 a b^3 +
 9 b^4}{108 (a + b) (3 a + 2 b) (2 a^2 + 4 a b + b^2)}$ \\
 \hline
 $5$ \T \B & $\frac{64 a^5 + 288 a^4 b + 452 a^3 b^2 + 298 a^2 b^3 + 80 a b^4 +
 7 b^5}{(8 a + 5 b) (4 a^2 + 6 a b + b^2) (4 a^2 + 10 a b + 5 b^2)}$ & $\frac{5 (2 a + b) (16 a^4 + 60 a^3 b + 73 a^2 b^2 + 32 a b^3 +
   4 b^4)}{24 (8 a + 5 b) (4 a^2 + 6 a b + b^2) (4 a^2 + 10 a b +
   5 b^2)}$\\
 \hline
\end{tabular}
\end{center} \caption{For pm-graph $\ga=L_n(a,b)$, $\ed$ and $Z(\ga)$ for $n \in \{ 2, \, 3, \, 4, \, 5 \}$.} \label{table epsilon ladder}
\end{table}
Again we used Mathematica \cite{MMA} to do exact computations given in \tabref{table exact ladder}.
\begin{table}
\begin{center}
\begin{tabular}{|c|c|c|c|c|c|c|}
  \hline
 $n$ \T \B & $\frac{\tg}{\elg}$ & $\frac{\tcg}{\elg}$ & $\frac{\vg}{\elg}$ &  $\frac{\lag}{\elg}$ & $\frac{\ed}{\elg}$ & $\frac{Z(\ga)}{\elg}$ \\
 \hline
 $5$ \T \B & $\frac{661}{10868}$ & $\frac{5546}{2717}$ & $\frac{411}{2717}$ & $\frac{5}{39}$ & $\frac{1189}{2717}$ & $\frac{925}{21736}$\\
 \hline
 $10$ \T \B & $\frac{2107}{37829}$ & $\frac{554308}{37829}$ & $\frac{30329}{71676}$ & $\frac{15}{76}$ & $\frac{344578}{340461}$ & $\frac{210215}{6128298}$\\
 \hline
 $15$ \T \B & $\frac{3061011619}{56529128700}$ & $\frac{180955287578}{4710760725}$ & $\frac{2384321993}{3411240525}$ & $\frac{665}{2494}$& $\frac{155613041207}{98925975225}$ & $\frac{2950668709}{92330910210}$\\
 \hline
 $20$ \T \B & $\frac{105284865781}{1971566979888}$ & $\frac{6020905705851}{82148624162}$ & $\frac{12183994532757}{12486590872624}$ & $\frac{380}{1131}$ & $\frac{6652614900537}{3121647718156}$ & $\frac{10979128575725}{355867839869784}$ \\
 \hline
 %$25$ & 13 & 10 & 4 & $8 a + 5 b$& &\\
 %\hline
\end{tabular}
\end{center} \caption{For pm-graph $\ga=L_n(1,1)$, computations of $\tg$, $\tcg$, $\vg$, $\lag$, $\ed$ and $Z(\ga)$ when $n \in \{ 5, \, 10, \, 15, \, 20 \}$.} \label{table exact ladder}
\end{table}

Finally, we used Matlab \cite{mat}, which does machine arithmetic, to obtain the results in \tabref{table inexact ladder}. Note that
when $n=25000$ in \tabref{table inexact ladder}, $L_n(1,1)$ has $50000$ vertices and $75000-2$ edges. As such computations would be possible on a computer with high memory and processing speed, we used Mac Pro with processor $2 \times 2.93$ GHz $6$-core Intel Xeon ($24$ hyper-threading in total) and memory $32$ GB $1333$ MHz DDR3 to obtain these results.
\begin{table}
\begin{center}
\begin{tabular}{|c|c|c|c|c|c|c|}
  \hline
 $n$ \T \B & $\frac{\tg}{\elg}$ & $\frac{\tcg}{\elg}$ & $\frac{\vg}{\elg}$ &  $\frac{\lag}{\elg}$ & $\frac{\ed}{\elg}$ & $\frac{Z(\ga)}{\elg}$ \\
 \hline
 $500$ \T \B & $0.05134130$ & $55155.801168$ & $27.6397$ & $7.00233$ & $55.4713$ & $0.0278941$\\
 \hline
 $1000$ \T \B & $0.05129978$ & $221422.325273$ & $55.4174$ & $13.9468$ & $111.027$ & $0.0278359$\\
 \hline
 $5000$ \T \B & $0.05126659$ & $5551550.777737$ & $277.639$ & $69.5023$ & $555.471$ & $0.0277894$\\
 \hline
 $10000$ \T \B & $0.05126245$ & $22214198.121181$ & $555.417$ & $138.947$ & $1111.03$ & $0.0277836$\\
 \hline
 $15000$ \T \B & $0.05126107$ & $49987945.991478$ & $833.194$ & $208.391$ & $1666.58$ & $0.0277816$\\
 \hline
 $20000$  \T \B & $0.05126038$ & $88872829.563811$ & $1110.97$ & $277.835$ & $2222.14$ & $0.0277807$\\
 \hline
$25000$  \T \B & $0.05125992$ & $138868444.360781$ & $1388.75$ & $347.279$ & $2777.69$ & $0.02778$\\
 \hline
\end{tabular}
\end{center} \caption{For pm-graph $\ga=L_n(1,1)$, computations of $\tg$, $\tcg$, $\vg$, $\lag$, $\ed$ and $Z(\ga)$ when $n \in \{ 500, \, 1000, \, 10000, \, 15000, \, 20000, \, 25000 \}$.} \label{table inexact ladder}
\end{table}

Symbolic computations are clearly the most costly computations. Among the numerical computations, we can have exact arithmetic (arithmetic with numbers having an infinite number of significant figures, i.e, with numbers having infinite precision), precision arithmetic (which involves the numbers with precision more than $18$) and arithmetic with machine numbers.

%\newpage

\textbf{Acknowledgements:} This work is supported by The Scientific and Technological Research Council of Turkey-TUBITAK (Project No: 110T686).

%$\cite{C1}$

\end{document}